\documentclass[11pt]{amsart}
\usepackage{amssymb}
\usepackage{amsmath}
\usepackage{graphicx}
\usepackage{hyperref}
\usepackage{xcolor}
\usepackage{soul}

\newtheorem{thm}{Theorem}[section]
\newtheorem{lem}[thm]{Lemma}

\newtheorem{cor}[thm]{Corollary}

\newtheorem{Def}[thm]{Definition}
\newtheorem{prop}[thm]{Proposition}
\newtheorem{rem}[thm]{Remark}
\newtheorem{ex}[thm]{Example}

\newcommand{\bdfn}{\begin{Def} \rm}
\newcommand{\edfn}{\end{Def}}

\newcommand{\lgra}{\longrightarrow}

\newcommand{\sm}{\setminus}

\newcommand{\beqa}{\begin{eqnarray*}}
\newcommand{\eeqa}{\end{eqnarray*}}

\def\nn{\|\cdot\|}
\newcommand{\tnorm}{|\hskip-.07em\|}
\def\nnn{\tnorm\cdot\tnorm}

\def\ee{\varepsilon}

\def\aa{\alpha}

\def\ll{\lambda}

\def\dd{\delta}

\def\Ext{\, {\rm Ext}\, }

\def\N{\Bbb N}
\def\IN{\hbox{{\rm I}\kern-.13em{\rm N}}}

\def\IR{\hbox{{\rm I}\kern-.13em{\rm R}}}
\def\nin{n\in\N}

\newcounter{cnt1}
\newcounter{cnt2}
\newcounter{cnt3}
\newcounter{cnt4}
\newcommand{\blr}{\begin{list}{$($\roman{cnt1}$)$} {\usecounter{cnt1}
 \setlength{\topsep}{0pt} \setlength{\itemsep}{0pt}}}
\newcommand{\blR}{\begin{list}{\Roman{cnt4}.\ } {\usecounter{cnt4}
 \setlength{\topsep}{0pt} \setlength{\itemsep}{0pt}}}
\newcommand{\bla}{\begin{list}{$(\alph{cnt2})$} {\usecounter{cnt2}
 \setlength{\topsep}{0pt} \setlength{\itemsep}{0pt}}}
\newcommand{\bln}{\begin{list}{$($\arabic{cnt3}$)$} {\usecounter{cnt3}
 \setlength{\topsep}{0pt} \setlength{\itemsep}{0pt}}}
\newcommand{\el}{\end{list}}

\begin{document}
\title[\tiny{A study of weak$^*$-weak points of continuity in Banach spaces}]{A study of weak$^*$-weak points of continuity in the unit ball of dual spaces}
\author[Daptari, Montesinos, Rao]{S. Daptari, V. Montesinos, and T. S. S. R. K. Rao}

\address{\textbf{Current address of Dr. Daptari:} Katsushika Division, Institute of Arts and Sciences, Tokyo University of Science, Tokyo 125-8585, Japan}
\email{daptarisoumitra@gmail.com}

\address{Shiv Nadar IoE. Gautam Buddha Nagar-201314, India}
\email{srin@fulbrightmail.org}
\address{Vicente Montesinos. Departamento de Matem\'atica Aplicada. Universitat Polit\`ecnica de Val\`encia, Camino de Vera s/n, 46071 Valencia, Spain}
\email{vmontesinos@mat.upv.es}
\subjclass[2020]{Primary 46A22, 46B10, 46B25, 47L40; Secondary 46B20, 46B22. \hfill
\textbf{\today}}
\keywords{Unique Hahn--Banach extension property, weak$^*$-weak points of continuity, $G_{\delta}$ sets of points of continuity, $M$-ideals, geometry of higher duals of Banach spaces, $C^*$-algebras, von Neumann algebras}
\maketitle
\begin{abstract}We study classes of Banach spaces where the points of weak$^*$-weak continuity for the identity mapping on the dual unit ball form a weak$^*$-dense and weak$^*$-$G_{\delta}$ set. We also discuss how this property behaves in higher duals of Banach spaces. We prove in particular that if $\mathcal{A}$ is a von Neumann algebra and its predual has the Radon--Nikod\'ym property, then there is no point of weak$^*$-weak continuity on the unit ball of $\mathcal{A}$.
\end{abstract}
\textbf{To appear in Revista de la Real Academia de Ciencias Exactas, F\'{i}sicas y Naturales. Serie A. Matem\'{a}ticas.}
\maketitle
\section{Introduction}
A well-known  notion in the geometry of Banach spaces is that of a denting point of the closed unit ball $B(X)$ of a Banach space $(X,\nn)$. We recall from \cite{DU} that $x_0 \in B(X)$ is said to be a {\bf denting point} if for any $\epsilon>0$, $x_0 \notin \overline{\rm{co}}( B(X)\sm B(x_0,\epsilon))$, where `co' denotes the convex hull and $B(x_0,\ee):=\{x\in X:\ \|x-x_0\|<\ee\}$. This has been shown in \cite{LLT} to be equivalent to $x_0\in\Ext(B(X))$ (i.e., $x_0$ being an extreme point of $B(X)$) and a point of continuity of the identity mapping $id:(S(X),w) \longrightarrow (S(X),\|.\|)$, where $S(X)$ denotes the unit sphere of $X$ and $w$ stands for the weak topology. In the case of a dual space, one considers the notion of $w^*$-denting points, where the $w$-topology gets replaced by the $w^*$-topology. An interesting question is to consider only points of continuity (not necessarily extreme points) and the geometric implications related to the ``richness'' of these points. We refer to the monograph \cite{Ho} for basic results in  geometry of Banach spaces.

Our work is motivated by \cite{HL}, where the authors study points of continuity of $id: (B(X^*),w)\longrightarrow (B(X^*),\nn)$ and $id: (B(X^{*}),w^*)\longrightarrow (B(X^{*}),\|.\|)$. We focuss instead on the study of points of continuity of $id: (B(X^{*}),w^*)\longrightarrow (B(X^{*}),w)$ (we shall speak of ``points of $w^*$-$w$-continuity'', and ---needless to say--- we are interested only in the nonreflexive case).

\medskip

Our notation is standard. Along the paper, $(X,\nn)$ will be a Banach space (we shall write just $X$ is there is no need to refer to the underlying norm). ``Equality'' of Banach spaces is understood as ``linear isometric equality''. As usual, we consider a Banach space $X$ canonically embedded in its bidual $X^{**}$. Similar canonical identifications are done for $X^*\subset X^{***}$, $X^{**}\subset X^{(IV)}$ and so on (higher duals are denoted as expected). The symbol $\{x_n\}$ will denote a sequence as well as a set, hoping that there will be no misunderstanding; the symbol $(x_n)$ will be used for vectors in a vector space.

\section{An interlude: $B(X^*)$ versus $S(X^*)$}
Despite the elementary nature of this section, we find it worth to clarify the reason why the statement of Lemma \ref{lemma-godefroy} below and many other similar results in this paper, although dealing with $id: (B(X^{*}),w^*)\longrightarrow (B(X^{*}),w)$, focus only on points in $S(X^*)$.

\medskip

We start with the following lemma, a $w^*$-$w$-version of \cite[Lemma 2.1]{HL}:
\begin{lem}\label{l22}
    Let $x^*\in B(X^*)$ be a point of continuity of $id: (B(X^{*}),w^*)\longrightarrow (B(X^{*}),w)$. Assume that $x^*=\lambda y^*+(1-\lambda)z^*$, where $y^*,z^*\in B(X^*)$ and $\lambda\in (0,1)$. Then both $y^*$ and $z^*$ are points of continuity of $id: (B(X^{*}),w^*)\longrightarrow (B(X^{*}),w)$.
\end{lem}
\begin{proof}Let $\{y^*_{\alpha}\}\subset B(X^*)$ be a net such that $y^*_{\alpha} \stackrel{w^*}\rightarrow y^*$. Clearly $\lambda y^\ast_{\alpha}+ (1-\lambda)z^\ast\in B(X^*)$ for all $\aa$,  and $\lambda y^\ast_{\alpha}+ (1-\lambda)z^\ast \stackrel{w^*}\rightarrow x^*$. By hypothesis, $\lambda y^\ast_{\alpha}+ (1-\lambda)z^\ast \stackrel{w}\rightarrow x^*\ (=\ll y^*+(1-\ll)z^*)$. It follows that $y^\ast_{\alpha} \stackrel{w}\rightarrow y^\ast$. The same argument applies to $z^*$, and the proof is over.
\end{proof}
The following is a straightforward consequence of the previous lemma, and locates the points of $w^*$-$w$-continuity of $B(X^*)$ on $S(X^*)$ in the nonreflexive case:
\begin{lem}\label{lemma-sphere} Let $X$ be a nonreflexive Banach space, and let $x^*\in B(X^*)$ be a point of continuity of $id: (B(X^{*}),w^*)\longrightarrow (B(X^{*}),w)$. Then, $x^*\in S(X^*)$.
\end{lem}
\begin{proof}

      Assume $\|x^*\|<1$. Let $(x^*\not=)\  y^*\in B(X^*)$. Obviously, $x^*=\ll y^*+(1-\ll)z^*$ for some $z^*\in B(X^*)$ and $\ll\in(0,1)$. Lemma \ref{l22} shows that $y^*$ is a point of continuity of $id:(B(X^*),w^*)\longrightarrow (B(X^*),w)$. It follows that $w$ and $w^*$ agree on $B(X^*)$, hence $X$ is reflexive, a contradiction.
\end{proof}
The next lemma is an almost trivial consequence of the $w^*$-lower-semicontinuity of the dual norm of a Banach space, and thus its proof will be omitted.
\begin{lem}\label{l24}
A point $x^*\in S(X^*)$ is of continuity of $id: (B(X^{*}),w^*)\longrightarrow (B(X^{*}),w)$ if and only if it is a point of continuity of $id: (S(X^{*}),w^*)\longrightarrow (S(X^{*}),w)$.
\end{lem}

\section{Introduction (continued) and some definitions}\label{sect-intro-bis}

In this paper, we also study the geometric property of points of $w^*$-$w$-continuity in the unit sphere of higher odd dual of a nonreflexive Banach space. We use renorming techniques to exhibit points of $w^*$-$w$-continuity in $B(X^*)$ that do not remain as such in $B(X^{***})$ (see Theorem~\ref{thm-star-not-3star} below). In the other direction (i.e., existence of points of $w^*$-$w$-continuity in higher duals that do not remain such when ``going down''), we exhibit some ``dramatic'' situations in spaces $L^1(\mu)$ of integrable functions, where (in the non-atomic case) there are no such points in the second dual unit ball, yet they do exist when passing to the fourth dual (Theorem \ref{thm-4-yes-2-no}).

\medskip

For every Banach space $X$, we have $X^{***}= X^*\oplus X^{\perp}$, where $X^{\perp}$ is the annihilator of $X$ (in $X^{***}$). Recall some geometric notions from \cite{HWW}.
A closed subspace $Y$ of a Banach space $X$ is said to be an {\bf $M$-ideal in $X$} if there exists a linear projection $P:X^*\longrightarrow X^*$ with $\ker P=Y^\perp$ and $\|x^*\|=\|Px^*\|+\|(I-P)x^*\|$ for all $x^*\in X^*$ (such a projection is called an {\bf $L$-projection}). A Banach space $X$ is said to be an {\bf $M$-embedded} space if under the canonical embedding, it is an $M$-ideal in $X^{**}$. It is equivalent to say that $X^{***}=X^{*}\oplus_1X^{\perp}$ \cite[Proposition III.3.2]{HWW}. See \cite[Chapters III and VI]{HWW} for several examples of those concepts  among function spaces and spaces of operators. A projection $P: X \longrightarrow X$ is said to be an {\bf $M$-projection} if $\|x\|=\max \{\|Px\|,\|x-Px\|\}$ for all $x\in X$. The range of an $M$-projection is called an {\bf $M$-summand} (see \cite[Definition~I.1.1]{HWW}). Let ${\mathcal L}(X)$, ${\mathcal K}(X)$ denote the space of bounded and compact operators on $X$, respectively. A particularly interesting situation occurs when ${\mathcal K}(X)$ is an $M$-ideal in ${\mathcal L}(X)$. In this case, the $w^*$, $w$, and norm topologies coincide on $S(X^*)$ (see Chapter VI and Chapter III of \cite{HWW}).
One of our main results implies that when $X$ is a nonreflexive Banach space such that ${\mathcal K}(X)$ is an $M$-ideal in ${\mathcal L}(X)$,  these points of continuity in $B(X^{***})$ form a weak$^*$-dense and weak$^*$-$G_{\delta}$ set  in $S(X^{***})$ (see Theorem~\ref{5t3}).

\medskip

Let $Y$ be a closed subspace of a Banach space $X$. For $x\in X$ we denote
\begin{equation}
P_Y(x):=\{y\in Y:\|x-y\|=d(x,Y)\}
\end{equation}
the {\bf metric projection} of $x$ into $Y$. The set $P_Y(x)$  can be empty. If for every $x\in X$ we have $P_Y(x)\not=\emptyset$, then $Y$ is said to be a {\bf proximinal} subspace. If, moreover, $P_Y(x)$ is a singleton for all $x\in X$, then $Y$ is said to be a {\bf Chebyshev} subspace.

A closed subspace $Y$ of $X$ is said to have {\bf property-$(U)$} if every bounded linear functional on $Y$ has a unique norm-preserving extension to $X$. Property-$(U)$ was introduced by R. R. Phelps in \cite[Theorem~1]{P}. According to him, $Y$ has property-$(U)$ if and only if $Y^{\perp}$ is a Chebyshev subspace of $X^*$. See also the recent work \cite{C} for a study of subspaces of $C(K)$ spaces with property-$(U)$, where $C(K)$ denotes the set of all continuous functions on a compact Hausdorff space K, equipped with the supremum norm. If we consider $X$ as a subspace of $X^{**}$, then $X$ is said to be {\bf Hahn--Banach smooth} if $X$ has property-$(U)$ in $X^{**}$.

\medskip

Property-$(U)$ of $X$ and the $w^*$-$w$-continuity of the identity mapping $id:B(X^*)\longrightarrow B(X^*)$ are closely related: It is enough to observe the following result, an application of Goldstine  theorem on the $w^*$-denseness of $B(X^*)$ in  $B(X^{***})$ and the $w^*$-compactness of $B(X^{***})$:
\begin{lem}\label{lemma-godefroy}\cite{G}
    Let $X$ be a Banach space. Then, $x^*\in S(X^*)$ is a point of continuity of $id: (B(X^{*}),w^*)\longrightarrow (B(X^{*}),w)$ if and only if $x^*$ has a unique norm-preserving extension to $X^{**}$.
 \end{lem}

The following theorem summarizes some of the results mentioned in the preceding discussion (see also Remark \ref{rem-local} below):
\begin{thm}\label{thm-hb-smooth}
For a Banach space $X$, the following are equivalent:
    \bln
    \item[\rm{(1)}] $id: (B(X^{*}),w^*)\longrightarrow (B(X^{*}),w)$ is continuous at all $x^*\in S(X^*)$ {\em (see Lemma \ref{lemma-sphere} above)}.
    \item[\rm{(2)}] $X$ is Hahn--Banach smooth in $X^{**}$.
    \item[\rm{(3)}] $X^{\perp}$ is a Chebyshev subspace of $X^{***}$.
    \el
\end{thm}
\section{Building points of discontinuity}
We recall from \cite{GIV} the notion of strong proximinality of a  subspace $Y \subset X$. Our definition is equivalent to the one given in \cite{GIV}. A proximinal subspace $Y \subset X$ is said to be {\bf strongly proximinal} if for $x\in X$ and any sequence $\{y_n\} \subset Y$ such that $d(x,Y) = \lim_n \|x-y_n\|$ (such a sequence is called a {\bf minimizing sequence for $x$}), there is a subsequence $\{y_{n_k}\}$ of $\{y_n\}$ and a sequence $\{z_{k}\} \subset P_Y(x)$ such that $\|y_{n_k}-z_{k}\| \rightarrow 0$.  We note that any minimizing sequence is a bounded sequence and that any subsequence (or subnet) of a minimizing sequence is again minimizing (the meaning of a {\bf minimizing net} should be clear). It is easy to see that any finite-dimensional subspace of a Banach space is strongly proximinal.

\medskip

The following theorem gives in particular, for a dual Banach space and for $\tau$ the $w^*$-topology, a procedure for getting vectors in $S(X^*)$ that are not of continuity for  $id:(B(X^*),w^*)\longrightarrow (B(X^*),\nn)$ (see Corollary \ref{cor-reflexive} below).
\begin{thm}\label{2.1}
Let $X$ be a Banach space. Let $Y$ be a Chebyshev not strongly proximinal subspace of $X$. Let $\tau$ be a locally convex topology on $X$ such that $\|.\|$  is a $\tau$-lower semi-continuous function. Assume that, for every $x \notin Y$, every minimizing sequence in $Y$ for $x$ has a subnet that $\tau$-converges to an element in $Y$. Then, there is $z \in S(X)$ that is not a point of $\tau$-norm continuity for $id: (B(X),\tau) \longrightarrow (B(X),\|.\|)$.
\end{thm}
\begin{proof}
Since $Y$ is Chebyshev and not strongly proximinal, there exists $x\in X$ (notice that $P_Y(x)$ is a singleton, say $\{y_0\}$) and a sequence $\{y_n\}$ in $Y$ that is minimizing for $x$ and such that for every subsequence $\{y_{n_k}\}$ we have $\|y_{n_k}-y_0\|\not\rightarrow 0$. Thus, there exists $\ee>0$ and a subsequence of $\{y_n\}$ (denoted again $\{y_n\}$) such that $\|y_n-y_0\|\ge\ee$ for all $\nin$. Without loss of generality, we may and do assume that $(\|x-y_0\|=)\ d(x,Y)=1$.

By assumption, there exists a subnet $\{y_{n_i}\}$ of the sequence $\{y_n\}$ that $\tau$-converges to some $y'\in Y$. The lower semicontinuity assumption and again the fact that $Y$ is Chebyshev ensures that $y'=y_0$.

Put $z:=x-y_0$ and $z_i:=x-y_{n_i}$ for all $i$. Then, $z_i\stackrel{\tau}\rightarrow z$, $\|z_i\|\rightarrow 1$, and $\|z_i-z\|\ge \ee$ for all $i$. It is obvious then that $\{z_i/\|z_i\|\}\stackrel{\tau}{\rightarrow} z$ and $\{z_i/\|z_i\|\}\not\rightarrow 0$ in norm. This shows the statement.
\end{proof}
\begin{cor}\label{cor-reflexive}
Let $X$ be a Banach space and let $Y$ be a $w^\ast$-closed Chebyshev not strongly proximinal subspace of $X^*$. Then, there is $z\in S(X^*)$ that is not a point of continuity of $id:(B(X^*),w^*)\longrightarrow (B(X^*),\nn)$.
\end{cor}
\begin{proof}
The fact that $Y$ is $w^*$-closed, that every minimizing sequence is bounded, and the $w^*$-compactness of $B(Y)$, ensure that the hypothesis of Theorem \ref{2.1} are satisfied for $\tau$ the $w^\ast$-topology.
\end{proof}

\begin{thm}
 Let $X$ be an infinite-dimensional reflexive Banach space. Then, there is a reflexive Banach space $Z$ containing $X$ as a subspace of codimension $1$, such that $Z$ has a unit vector that is not even a point of sequential continuity of $id: (B(Z),w) \longrightarrow (B(Z),\|.\|)$.
\end{thm}

\begin{proof}
By \cite[Theorem 2.2]{N}, there exists a Banach space $Z\supset X$ such that $X$ is of codimension $1$ in $Z$ (hence $Z$ is also reflexive), and $X$ is a Chebyshev not strongly proximinal subspace of $Z$. The rest is a consequence of Theorem \ref{2.1}, and the sequential continuity assertion follows just from the reflexivity of $Z$ and the Eberlein--\v Smulyan theorem.
\end{proof}

\section{Points of continuity in higher duals (I)}\label{sect-higher}

{\bf  In this and the following sections, we will work on nonreflexive Banach space.}

\subsection{Extreme points that are of continuity}\label{subsec-extreme}

An open problem in this theory is to find conditions under which the existence of a point of continuity of $id: (B(X^{*}),w^*)\longrightarrow (B(X^{*}),w)$ implies the existence of {\em extreme} points that are also points of continuity of $id: (B(X^{*}),w^*)\longrightarrow (B(X^{*}),w)$. This question will be considered in Section \ref{sect-g-delta}.

\medskip

We record the following simple lemma for later reference. Here, $NA(X^{**})$ denotes the set of elements in $X^{**}$ that attain their norm on $S(X^*)$:
\begin{lem}\label{lemma-preserved}
Let $X$ be a Banach space, and let $x^{**}\in B(X^{**})$ be a point of continuity of $id: (B(X^{**}),w^*)\longrightarrow (B(X^{**}),w)$. Then
\begin{enumerate}
\item $x^{**}\in S(X)$.
\item $x^{**}\in\Ext B(X^{**})$ if and only if $x^{**}\in\Ext B(X)$.
\item $x^{**}\in NA(X^{**})$.
\end{enumerate}
\end{lem}
\begin{proof}
Item (1) is a consequence of the $w^*$-density of $B(X)$ in $B(X^{**})$ and Lemma \ref{lemma-sphere} above. For Item (2), observe that an extreme point of $B(X^{**})$ that lies in $X$ is an extreme point of $B(X)$. In general, the converse is not true, see Remark \ref{r2} below; however, in case that $x\in\Ext B(X)$ is of continuity of  $id: (B(X^{**}),w^*)\longrightarrow (B(X^{**}),w)$, it is also in $\Ext(B_{X^{**}})$, a consequence of Lemma \ref{l22} above. Item (3) follows from (1).
\end{proof}

\begin{rem}\label{r2}
\rm It is worth to mention that $\Ext B(X^{**})=\Ext B(X)$ if and only if the space $X$ is reflexive, a consequence of James' characterization of reflexivity. Extreme points of $B(X)$ that are also extreme points of $B(X^{**})$ are called {\bf preserved extreme} (see, e.g., \cite{gmz} for an account of such a concept).
\end{rem}

Let $X$ be a Banach space such that	$X^{\ast\ast}/X$ is a reflexive space. Spaces with this property are analogues of quasi-reflexive spaces (Civin and Yood, \cite{civin-yood}). We recall that $X$ is said to be a {\bf quasi-reflexive space} if $X^{\ast\ast}/X$ is a finite-dimensional space.

\medskip

\begin{prop}
Let $X$ be a Banach space such that $X^{\ast\ast}/X$ is reflexive. Let $x^\ast \in S(X^\ast)$ be a point of continuity for $id: (B(X^\ast),w^\ast) \longrightarrow (B(X^\ast),w)$. Then it continues to be a point of $w^\ast$-$w$-continuity of the unit balls of all odd higher duals of $X$. In particular, if $x^*$ is also an extreme point, then it is a  preserved extreme point with respect to those odd higher duals.
\end{prop}

\begin{proof} As we mentioned several times, $X^{***}=X^*\oplus X^{\perp}$. Notice that the topology $w$ in $X^{***}$ induces the topology $w$ on $X^*$. Let $Q: X^{\ast\ast\ast}\rightarrow X^\ast$ be the canonical projection associated to this decomposition. Let $\{\tau_{\alpha}\} \subset S(X^{\ast\ast\ast})$ be a net such that $\tau_{\alpha}\rightarrow x^\ast$ in $(B(X^{\ast\ast\ast}),w^*)$. Hence $Q(\tau_{\alpha}) \rightarrow x^\ast$ in $(B(X^\ast),w^*)$. Since $x^*$ is a point of continuity for $id: (B(X^\ast),w^\ast) \longrightarrow (B(X^\ast),w)$ we have $Q(\tau_{\alpha}) \rightarrow x^\ast$ in $(B(X^\ast),w)$ (and so $Q(\tau_\aa)\rightarrow x^*$ in $(B(X^{***}),w)$).

We get, in particular, $(X^{\perp}\ni)\ \tau_{\alpha}-Q(\tau_{\alpha}) \rightarrow 0$ in $(X^{\ast\ast\ast},w^*)$. Observe that $(X^{**}/X)^*$ is isomorphic to $X^{\perp}$, hence $X^{\perp}$ is also reflexive. Thus, $\tau_{\alpha} - Q(\tau_{\alpha}) \rightarrow 0$ in $(X^{\ast\ast\ast},w)$. Therefore $\tau_{\alpha} \rightarrow x^\ast$ in $(B(X^{\ast\ast\ast}),w)$ and so $x^*$ is of continuity of $id:(B(X^{***}),w^*)\longrightarrow (B(X^{***}),w)$.

Since $X^{***}=X^*\oplus X^{\perp}$, the space $X^{***}/X^*$ is reflexive, too, since it is isomorphic to $X^{\perp}$. Now we consider the duality $\langle X^{(IV)},X^{***}\rangle$. The dual space of $X^{***}/X^*$ is $(X^*)^{\perp}$, so this last space is also reflexive. Since $X^{(IV)}=X^{**}\oplus (X^*)^{\perp}$, and then $X^{(IV)}/X^{**}$ is isomorphic to $(X^*)^{\perp}$, it follows that $X^{(IV)}/X^{**}$ is reflexive and we can apply iteratively what has already been proved to reach the conclusion.

The assertion about extreme points follows from the first part of the statement and Lemma \ref{lemma-preserved}.
\end{proof}

We shall resume the study of extreme points of continuity in Section \ref{sect-extreme-2} below.

\subsection{Continuity in higher duals} Related to Lemma \ref{lemma-preserved}, let us mention that we don't know of any geometric condition on $x\in S(X)$ in terms of $X$ alone for ensuring that $x$ should be a point of continuity of $id: (B(X^{**}),w^*)\longrightarrow (B(X^{**}),w)$.

We refer the reader to what was said at the beginning of Section \ref{sect-intro-bis} above regarding the bad stability behaviour ---out of some classes of Banach spaces, see Section \ref{sect-m-hbs} for some instances---of points of $w^*$-$w$-continuity when going to higher duals and, viceversa, ``coming down'', and the appropriate examples mentioned there.

\section{M-embedded and Hahn--Banach smooth spaces}\label{sect-m-hbs}
Some properties of a Banach space remain when going to higher duals. For example, if $\mathcal{A}$ is a $C^*$-algebra, then $\mathcal{A}^{**}$ is again a $C^*$-algebra. Similarly, if $X$ is an abstract $L$-space or $M$-space, then their biduals are of the same type (see, e.g.,  \cite{L}). Thus it is an interesting question to see how some of the geometric properties of a Banach space studied here can be carried to its higher duals.

\medskip

Theorem \ref{thm-hb-smooth} above says, in particular, that for a Hahn--Banach smooth Banach space $X$, every $x^*\in S(X^*)$ is a point of continuity of $id:(B(X^*),w^*)\longrightarrow (B(X^*),w)$. We shall show in Theorem \ref{t31} that, under the slightly stronger assumption of $M$-embeddability (see Remark \ref{rem-m-implies-hbs} below), it is even of continuity for $id:(B(X^{***}),w^*)\longrightarrow (B(X^{***}),w)$.

\begin{rem}\label{rem-m-implies-hbs}\rm
{\em Every $M$-embedded Banach space is Hahn--Banach smooth}. More generally, if $M$ is an $M$-ideal in a Banach space $Y$, then every $m^*\in M^*$ has a unique norm-preserving extension to a functional $y^*\in Y^*$ (see, e.g., \cite[Proposition 1.12]{HWW}). As a consequence, $M^*$ can be viewed as a subspace of $Y^*$ and $Y^*= M^* \oplus_1 M^\bot$(see, e.g.,  \cite[Remark I.1.13]{HWW}).
\end{rem}

\begin{thm}\label{t31}
Let $X$ be a Banach space, and let $M\subset X$ be an $M$-ideal in $X$. Then every $m^*\in S(M^*)$ that is a point of continuity  for $id:(B(M^*),w^*)\longrightarrow (B(M^*),w)$, is a point of continuity for  $id:(B(X^*),w^*)\longrightarrow (B(X^*),w)$. In particular, if $X$ is an $M$-embedded space, then the set of points of continuity for $id:(B(X^{***}),w^*)\longrightarrow (B(X^{***}),w)$ is precisely $S(X^*)$.
\end{thm}
\begin{proof}
As we mentioned in Remark \ref{rem-m-implies-hbs} above, $M^*$ can be identified with a subspace of $X^*$, and so the conclusion makes sense. Let $Q:X^*\longrightarrow X^*$ be an $L$-projection with range $M^{\perp}$. Put $P:=(I-Q)$ so that $range(P)=M^*$. Let $(x^*_i)$ be a net in $B(X^*)$ that $w^*$-converges to $m^*$. Then, $P(x^*_i)\stackrel{w^*}{\rightarrow }P(m^*)\ (=m^*)$.  Thus, by hypothesis, $P(x^*_i)\stackrel{w}\rightarrow P(m^*)\ (=m^*)$ weakly. Notice that $x^*_i=P(x^*_i)+(I-P)(x^*_i)$, and that $\|P(x^*_i)\|\rightarrow 1$. The fact that $P$ is an $L$-projection with range $M^*$ and kernel $M^{\perp}$ shows that $\|(I-P)(x^*_i)\|\rightarrow 0$. The conclusion follows.

For the particular case, apply the result to the space $X^{**}$ and its $M$-embedded subspace $X$ to get that every $x^*\in S(X^*)$ is of continuity for $id:(B(X^{***}),w^*)\longrightarrow (B(X^{***}),w)$. In the other direction, every point of continuity for $id:(B(X^{***}),w^*)\longrightarrow (B(X^{***}),w)$ is in $S(X^*)$, according to Lemma \ref{lemma-preserved} above.
\end{proof}

As an application, we have the following
\begin{ex}\rm
    Let $1<p<\infty$ and $q$ be such that $\frac{1}{p}+\frac{1}{q}=1$. It is known that ${\mathcal K}(\ell^p)^*=\ell^p\hat{\otimes}_{\pi}\ell^q$, where $\ell^p\hat{\otimes}_{\pi}\ell^q$ is projective tensor product of $\ell^p$ and $\ell^q$.  It is also well known that ${\mathcal K}(\ell^p)^{**}={\mathcal L}(\ell^p)$. One also has that ${\mathcal K}(\ell^p)$ is an $M$-ideal in its bidual (see \cite[Example.III.1.4]{HWW}). Thus, from Theorem~{\rm{\ref{t31}}} we get that all points of $S({\mathcal K}(\ell^p)^*)=S(\ell^p\hat{\otimes}_{\pi}\ell^q)$ are points of $w^*$-$w$-continuity of $B({\mathcal L}(\ell^p)^*)$.
\end{ex}

\begin{rem}\rm Theorem \ref{t31}, together with Lemma \ref{lemma-godefroy}, provide an alternative proof of the following result \cite[Proposition 3.2]{GI}: {\em Let $X$ be an $M$-embedded Banach space. If the norm is G\^ateaux smooth in $x\in S(X)$, then $x$ is a very smooth point of $X^{**}$}. For the concept of very smoothness, see the beginning of Section \ref{sect-higher-2} below.
\end{rem}

The rest of this section will address the case where $x\in S(X)$ is a point of continuity of $id: (B(X^{**}),w^*)\longrightarrow (B(X^{**}),w)$ (see Lemma \ref{lemma-preserved} above).
Let $Y$ be another Banach space such that $\Phi: X^{*}\longrightarrow Y^{*}$ is a surjective linear isometry. Then $\Phi^{*}: Y^{**}\longrightarrow X^{**}$ is a linear surjective isometry that is also an isomorphism for the $w^*$-topologies. It is easy to see that the set of points of $w^*$-$w$-continuity of the identity in $B(X^{**})\ (=B(Y^{**}))$ does not depend on the underlying $X$ or $Y$ space. In fact, this set lies in $X\cap Y$, where both $X$ and $Y$ are naturally seen as subspaces of $X^{**}$.

In several cases where this idea is applied in this paper, we have that $X^*$ has the RNP (the Radon--Nikod\'ym property, see \cite{DU}). For example, if the space $X$ is Hahn-Banach smooth (in particular, if it is $M$-embedded), then $X$ is Asplund (and so $X^*$ has the RNP). In fact, we have more: Since the $w^*$ and the $w$ topology coincide on $S(X^*)$, then $X^*$ admits an equivalent dual LUR norm, and so $X$ has an equivalent Fr\'echet differentiable norm \cite{raja}. Then we may apply \cite[Theorem~5.7.6]{Bo} to conclude that $X^*$ is the unique predual of $X^{**}$, and that there is a unique norm-one projection of $X^{(IV)}$ onto $X^{**}$. Thus, if $Y$ is a Banach space such that $Y^{**}=X^{**}$, then there is a surjective isometry $\Phi:Y^*\to X^*$. In this case, the weak$^*$-topology of $X^{**}$ is uniquely determined. In this direction, we consider the following example:
\begin{ex}\label{ex-c}\rm
Let $c$ be the space of convergent sequences, endowed with the supremum norm $\nn_\infty$. Notice that $c^*$ is linearly isometric to $\ell_1$ via the mapping $S: c^\ast  \longrightarrow \ell^1$ defined by $S(\alpha) = (\alpha({\bf 1})-\sum_1^\infty \alpha(e_n), \alpha(e_1),\alpha(e_2),...)$ for $\aa\in c^*$, where the $e_i$'s are the elements in $c$ with $1$ at the $i$th place and $0$ elsewhere, and ${\bf 1}$ is the constant-$1$ sequence (see, e.g., \cite[Exercise 2.31]{y2}).

Consider $(c,\nn_{\infty})$. Certainly,  $c_0$ is a closed subspace. Remark \ref{r2} says, in particular, that every point $x\in B(c^{**})$ of continuity for $id:B(c^{**},w^*)\longrightarrow B(c^{**},w)$ is in $c$. However, not every point of $c$ is such. In fact, no point in $c\setminus c_0$ is of continuity.

In order to see this, let $T:=S^{-1}:\ell_1\longrightarrow c^*$ (a linear isometry onto). Then $T^*:c^{**}\longrightarrow \ell_{\infty}$ is also a linear isometry onto. We shall translate the situation in $c^{**}$ via $T^*$ to the one in $\ell_{\infty}$ (by the way, as it is clear, $w(c^{**},c^*)$ corresponds via $T^*$ to $w(\ell_{\infty},\ell_1)$, and $w(c^{**},c^{***})$ to $w(\ell_{\infty},\ell_{\infty}^*))$. It is easy to compute $T^*(x)$, where $x=(x_n)\in c\ (\subset c^{**})$. We get $T^*(x)=(l,x_1,x_2,\ldots)\ (\subset \ell_{\infty})$, where $l:=\lim_nx_n$. Let $x=(x_n)\in c\setminus c_0$. Then $u:=T^*(x)=(l,x_1,x_2,\ldots)$, where $l\not=0$. Put $u_n:=(l,x_1,x_2,\ldots,x_n,0,0,\ldots)\ (\in c_0\subset \ell_{\infty})$ for $\nin$. Then, the sequence $\{u_n\}$ in $B(\ell_\infty)$ is $w(\ell_{\infty},\ell_1)$-convergent to $u$, due to the fact that on $B(\ell_{\infty})$ the topology $w(\ell_{\infty},\ell_1)$ and the one of coordinate-wise convergence coincide. Due to the fact that $(c_0)^{\perp\perp}=c_0$ in the duality $\langle \ell_{\infty},\ell_{\infty}^*\rangle$, there exists $u^*\in c_0^{\perp}\subset \ell_{\infty}^*$ such that $\langle u,u^*\rangle\not=0$. Thus, $(0=)\ \langle u_n,u^*\rangle\not\rightarrow \langle u,u^*\rangle\ (\not=0)$. This shows that $\{u_n\}\not\rightarrow u$ in the topology $w(\ell_{\infty},\ell_{\infty}^*)$. Since $u=T^*(x)$ for $x\in c\setminus c_0$, it follows that $x$ is not a point of $w^*$-$w$-continuity for the identity mapping $id:(B(c^{**}),w^*)\longrightarrow (B(c^{**}),w)$, as announced.
\end{ex}
It is known that $\ell^1$ is the unique predual of $\ell^\infty$ and that $c_0$ is an $M$-embedded space (in $\ell^{\infty}$) (see, e.g., \cite[Examples 1.4 (a)]{HWW}). The concrete example above (Example \ref{ex-c}) has a very general geometric analogue.  We need the  following definition (see, e.g., \cite{Go87}): A Banach space $X$ is said to be a {\bf strongly unique predual} (of $X^*$) if, for every Banach space $Y$, every isometric isomorphism from $X^*$ onto $Y^*$ is $w^*$-continuous (and so it is the adjoint of an isometric isomorphism from $Y$ onto $X$). This is an (apparently) strengthening of the notion of a {\bf unique predual} (if $Y$ is a Banach space such that $Y^*$ and $X^*$ are isometrically isomorphic, then $X$ and $Y$ are isometrically isomorphic, too; it seems that the problem of the coincidence of both notions is still open). Let us recall the following result:

\begin{prop}\label{prop-strong-1}\cite[Proposition III.2.10]{HWW}
Let $X$ be an $M$-embedded nonreflexive Banach space. Then, (i) $X^*$ is a strongly unique predual of $X^{**}$, and (ii) $X$ is not a strongly unique predual of $X^*$.
\end{prop}
An application of Proposition \ref{prop-strong-1} to the question of $w^*$-$w$-continuity in duals of nonreflexive $M$-embedded spaces is the following:
\begin{thm}\label{t3.5}
Let $X$ be a nonreflexive $M$-embedded space. Then, there exists a subspace $Y$ of $X^{**}$ such that $X^{**}=Y^{**}$ and $Y\sm X\neq \emptyset$. Consequently there is no point of continuity for $id: (B(X^{**}),w^*)\longrightarrow (B(X^{**}),w)$.
\end{thm}
\begin{proof} To organise the proof, we shall split the argument in several steps:

(a) Since $X$ is not a strongly unique predual of $X^*$ (see Proposition \ref{prop-strong-1} above), there exists a Banach space $Y$ and an isometric onto isomorphism $T:X^*\longrightarrow Y^*$ that is not $w^*$-$w^*$-continuous. Obviously, $Y^{**}$ is isometrically isomorphic to $X^{**}$ (and so $Y$ can be identified to $T^{*}(Y)\ (\subset X^{**})$). The fact that $T$ is not $w^*$-$w^*$-continuous shows that $T^*(Y)\setminus X\not=\emptyset$. Indeed, there exists a $w^*$-null net $\{x^*_i\}$ in $X^*$ such that $\{T(x^*_i)\}$ is not $w^*$-null in $Y^*$. Assume that $T^*(Y)\subset X$. Then, given $y\in Y$ we have $\langle y,T(x^*_i)\rangle=\langle T^*(y),x^*_i\rangle\rightarrow 0$. Thus, $T(x^*_i)\stackrel{w^*}{\rightarrow }0$, a contradiction. (Incidentally, that $T^*(Y)\setminus X\not=\emptyset$ is, in fact, {\em equivalent} to the non-$w^*$-$w^*$-continuity of $T$.) This proves the first part of the statement.

\medskip

(b) Let $x^{**}_0\in B(X^{**})$ be a point of continuity of $id:(B(X^{**}),w^*)\longrightarrow (B(X^{**}),w)$ (notice that, by Lemma \ref{lemma-sphere}, $x^{**}_0\in S(X^{**})$). It is easy to prove that $x^{**}_0\in X\cap T^*(Y)$. Indeed, by Lemma \ref{lemma-preserved}, $x^{**}_0\in S(X)$. The space $T^*(Y)$ is a predual of $X^{*}$. By Lemma \ref{lemma-preserved} again, $x^{**}_0\in T^*(Y)$.  (Another argument for the same conclusion is the following: First, it is simple to show that $B(T^*(Y))$ is $w^*$-dense in $B(X^{**})$. Thus, we can find a net $\{T^*(y_i)\}$ in $B(T^*(Y))$ that $w^*$-converges to $x^{**}_0$ and, by assumption, it is also $w$-convergent (to $x^{**}_0$). Since $T^*(Y)$ is $\nn$-closed in $X^{**}$, we get $x^{**}_0\in T^*(Y)$.) So we conclude that
\begin{equation}\label{eq-exclusion}
x^{**}_0\in X\cap T^*(Y).
\end{equation}

\medskip

(c) For the last part, we shall prove  the following

{\bf Claim}: {\em Given $x^{**}_0\in S(X^{**})$  a point of continuity of $id:(B(X^{**}),w^*)\longrightarrow (B(X^{**}),w)$ (and then $x^{**}_0\in S(X))$, the space $Y$ considered above can be chosen so that $x^{**}_0\not\in T^*(Y)$}.

To show the Claim, we need to consider the construction in the proof of \cite[Proposition III.2.10]{HWW}: There, a point $y^{***}\in X^{\perp}\ (=(X^{**}/X)^*)$ such that $\|y^{***}\|=1$ and that attains its norm at $\widehat{x^{**}}\in S(X^{**}/X)$ is selected, and the $w^*$-closed hyperplane $H:=\ker \widehat{x^{**}}$ of $X^{\perp}$ is constructed. Our observation is that, given an arbitrary $x_0\in S(X)$, we can find $p^*_0\in S(X^*)$ such that $\langle x_0,p^*_0\rangle=1$. This is used to define $Z:=H\oplus \mathbb R(y^{***}+sp^*_0)$ as in the original construction, where $s$ is an arbitrary real number in $(0,1]$. Thus $Z$ is a $w^*$ closed subspace of $X^{***}$. It was shown during the proof of \cite[Proposition III.2.10]{HWW}  that $X^{***}=X^*\oplus Z$, and that the canonical associated projection $\pi:X^{***}\rightarrow X^*$ associated to this decomposition is a contraction (here we need the $M$-embedded hypothesis again) . It is enough to observe now that $Z^{\perp}\ (\subset X^{**})$ (the orthogonal to $Z$  with respect to the duality $\langle X^{**},X^{***}\rangle$) is a space with the properties enjoyed by $T^*(Y)$ (the space mentioned at the beginning of the proof). It is clear then that $x_0\not\in Z^{\perp}$. This concludes the proof of the Claim.

(d) To finalize the proof of the theorem it is enough to use the Claim together with (b) (in particular, equation (\ref{eq-exclusion}) above).
\end{proof}

We now give an application to von Neumann algebras.
\begin{prop}
    Let $\mathcal{A}$ be a von Neumann algebra and the predual $X$ of $\mathcal{A}$ has the RNP. Then, there is no point of continuity of $id: (B(\mathcal{A}),w^*)\longrightarrow (B(\mathcal{A}),w)$.
\end{prop}
\begin{proof}
    By \cite[Proposition~IV.2.9]{HWW}, there exists a family $\{H_i\}_{i\in I}$ of Hilbert spaces such that $\mathcal{A}=\bigoplus_{\ell^\infty}\mathcal{L}(H_i)$ and $X=\bigoplus_{\ell^1}\mathcal{K}(H_i)^*$ (recall that $\mathcal{K}(H_i)^*$ is the space of all trace class operators with trace norm). Thus $\mathcal{A}=\left(\bigoplus_{c_0}\mathcal{K}(H_i)\right)^{**}$. It is known that $\bigoplus_{c_0}\mathcal{K}(H_i)$ is $M$-embedded (see \cite[Example~III.1.4]{HWW}). Hence the result follows from Theorem~\ref{t3.5}.
\end{proof}
\section{Geometry of  higher duals of Banach spaces}\label{sect-higher-2}
      Let $X$ be a nonreflexive Banach space. Let $x^\ast \in S(X^\ast)$ be a point of $w^\ast$-$w$-continuity. In general $x^\ast$ need not be a point of $w^\ast$-$w$-continuity in $S(X^{\ast\ast\ast})$.  We will be using \cite[Proposition~4.1]{GI} in what follows for a renorming technique to produce such examples (see Theorem \ref{thm-star-not-3star} below). For instances of spaces satisfying the hypothesis there, see, e.g.,  \cite[Theorem 1.d.3.]{LT}.

      \medskip

The concept of very smoothness lies in between G\^ateaux and Fr\'echet smoothness (see, e.g., \cite{Di}): Let $(X,\nn)$ be a Banach space. A point $0\not=x\in X$ is said to be a {\bf very smooth point} (or the norm of $X$ is {\bf very smooth at $x$}) if the norm of $X$ is G\^ateaux smooth at $x$ and the duality map from $X$ to $X^*$ is $\nn$-$w$-upper semi-continuous at $x$. The relevance to our topic is the equivalence between (i) and (iv) in \cite[Remark 3.1]{GI}, a result that we reproduce here for the sake of completeness:

\begin{lem}\cite[Remark 3.1]{GI}\label{lemma-godefroy-smooth}
If $x\in S(X)$ is a point of G\^ateaux smoothness, and $x^*=\nn'(x)$, then the following are equivalent:

(i) $x$ is a very smooth point of $X$.

(ii) The duality mapping is $\nn$-$w$-upper semicontinuous at $x$.

(iii) $x^*$ has a unique norm-preserving extension to $X^{**}$.

(iv) $id:(B(X^*),w^*)\longrightarrow (B(X^*),w)$ is continuous at $x^*$.

(v) $x$ is a point of G\^ateaux smoothness for $X^{**}$.

(vi) If $\{x^*_n\}$ is a sequence in $B(X^*)$ such that $\langle x,x^*_n\rangle\rightarrow 1$, then $x^*_n\stackrel{w}\rightarrow x^*$.
\end{lem}

\begin{rem}\label{rem-local}\rm
Observe that the equivalence $(iii)\Leftrightarrow(iv)$ in Lemma \ref{lemma-godefroy-smooth} holds for every point $x^*\in S(X^*)$, not necessarily the derivative of a point of G\^ateaux differentiability in $S(X)$. This was the content of Lemma \ref{lemma-godefroy}.
\end{rem}

\begin{thm}\label{thm-star-not-3star}
Let $X$ be a Banach space such that $X^{**}/X$ is a nonreflexive separable space. Then, there is a renorming of $X$, denoted by $Z$, and $x^*\in S(Z^*)$, such that $x^*$ is a point of continuity of $id: (B(Z^{*}),w^*)\longrightarrow (B(Z^{*}),w)$ but not a point of continuity of $id: (B(Z^{***}),w^*)\longrightarrow (B(Z^{***}),w)$.
\end{thm}
\begin{proof}
Assume first that $X$ is separable. In \cite[Proposition 4.1]{GI} it is proved that, in these circumstances, there exists an equivalent norm $\nnn$ on $X$ such that, denoting $Z$ the Banach space $(X,\nnn)$, we can find a very smooth point $x\in S(Z)$ which is not a very smooth point of $Z^{**}$.   Thus, it follows from Lemma \ref{lemma-godefroy} above that $x^*$ is a point of continuity for $id:(B(Z^{*}),w^*)\longrightarrow (B(Z^{*}),w)$, yet it is not a point of continuity for $id:(B(Z^{***}),w^*)\longrightarrow (B(Z^{***}),w)$. This proves the assertion in this particular case.

For the general case, notice that, by \cite{V}, the hypothesis implies that $X$ is the topological direct sum of a reflexive Banach space $R$ and a Banach space $S$ such that $S^{**}$ is separable. In particular, $S^{**}/S$ is separable (and nonreflexive as it is isomorphic to $X^{**}/X$). Thus, by what has been already proved, there exists an equivalent norm $\||.\||$ on $S$ and a point of continuity $f_0\in B((S,\||.\||)^*)$ of $id:(B((S,\||.\||)^*), w^*) \to (B((S,\||.\||)^*), w)$ that is not a point of continuity of $id:(B((S,\||.\||)^{***}), w^*) \to (B((S,\||.\||)^{***}), w)$. Let us define an equivalent norm (called again $\||.\||$) on $X$ by letting  $X=R\oplus_\infty S$. To simplify the notation, we denote $(X,\||.\||)$ by $Z$. Obviously, $Z^*=R^*\oplus_1 S^*$. Let $f\in Z^*$ be such that $f|_R=0$ and $f|_S=f_0$. By Theorem \ref{t31}, it is a point of continuity of $id:(B(Z^*), w^*)\longrightarrow (B(Z^*), w)$. However, the fact that two different norm-preserving extensions of $f_0$ from $S^{**}$ to $S^{(IV)}$ exists shows as before that $f$ is not a point of continuity of $id: (B(Z^{***}),w^*)\longrightarrow (B(Z^{***}),w)$.
\end{proof}

\section{Extreme points of continuity}\label{sect-extreme-2}
Since  $id: (B(X)^*, w^*) \longrightarrow (B(X)^*, w)$ is an affine map, it is natural to ask whether $id$ has an extreme point that is of continuity as soon as it has a point of continuity. This was mentioned already in Subsection \ref{subsec-extreme}.

\medskip

For any positive measure space $(\Omega,\Sigma, \mu)$ and for $1 \leq p <\infty$, let $L^p(\mu,X)$ be the space of Bochner integrable functions (for an authoritative reference see, e.g., \cite{DU}). In the case of points of $w$-$\nn$-continuity, for a non-atomic measure $\mu$ and for $1<p<\infty$, it was shown in \cite[Corollary 2.4]{HL} that $f$ is a point of $w$-$\nn_p$-continuity of $B_{L^p(\mu, X)}$ if and only if $f$ is a denting point of $B_{L^p(\mu, X)}$ (equivalently, an extreme point of $w$-$\nn_p$-continuity).  The next theorem treats the $w^*$-$w$-case in the same context.
\begin{thm}\label{4t2}
Let $(\Omega, \Sigma, \mu)$ be a positive measure space. If $\Lambda \in S(L^1(\mu))$ is a point of continuity of $id: (B(L^1(\mu)^{**}),w^*)\longrightarrow (B(L^1(\mu)^{**}),w)$, then $\mu$ has an atom. Hence $S(L^1(\mu))$ has an extreme point which is also a point of continuity of $id: (B(L^1(\mu)^{**}),w^*)\longrightarrow (B(L^1(\mu)^{**}),w)$.
\end{thm}
\begin{proof}
Let $\Lambda \in S(L^1(\mu))$.  It is known that $L^1(\mu)^*$, being a commutative $C^*$-algebra, is isometric to $C(K)$ for some compact Hausdorff space $K$. Thus, consider $\Lambda \in (L^1(\mu))^{**}=C(K)^*$, a point of continuity of $id: (B(L^1(\mu)^{**}),w^*)\longrightarrow (B(L^1(\mu)^{**}),w)$. It is not difficult to show that $\Lambda=\sum_{i=1}^{\infty}\alpha_i\delta(k_i)$, where the $k_i$'s are isolated points in $K$ (for those $i$ for which $\alpha_i \neq 0$) and $0\leq \alpha_i\leq 1$ with $\sum_{i=1}^{\infty}\alpha_i=1$. By \cite[Example.IV.1.1]{HWW}, $L^1(\mu)$ is an $L$-summand in $L^1(\mu)^{**}$. Let $P:L^1(\mu)^{**}\lgra L^1(\mu)$ be the corresponding $L$-projection. We claim that $\delta(k_i)\in L^1(\mu)$ when $\alpha_i\neq 0$. It is enough to show that $P(\delta(k_i))=\delta(k_i)$ for all $i$ such that $\alpha_i\neq 0$. Since $\Lambda \in S(L^1(\mu))$, $\|\Lambda\|=\|P(\Lambda)\|\leq \sum_{i=1}^{\infty} \alpha_i\|P(\delta(k_i))\|\leq 1$. Consequently, $\|P(\delta(k_i))\|=1$. Again, $\|\delta(k_i)\|=\|P(\delta(k_i))\|+\|\delta(k_i)-P(\delta(k_i))\|$ as $P$ is an $L$-projection. Thus $P(\delta(k_i))=\delta(k_i)$ for all $i$ such that $\alpha_i\neq 0$. Since $\alpha_i \neq 0$ for at least one $i$, one can show that the corresponding $\delta(k_i)$ gives raise to an atom for the measure $\mu$, as it is an extreme point of $B(L^1(\mu))$.
 \end{proof}
\begin{thm}\label{thm-4-yes-2-no}
Let $X=L^1[0,1]$, where $[0,1]$ carries the Lebesgue measure. Then $id: (B(X^{**}),w^*)\longrightarrow (B(X^{**}),w)$ is nowhere continuous, but $id: (B(X^{(IV)}),w^*)\longrightarrow (B(X^{(IV)}),w)$ is continuous at some extreme point in $B(X^{**})$.
\end{thm}
\begin{proof}
    Since the Lebesgue measure is non-atomic, by Theorem~\ref{4t2} we have that $id: (B(X^{**}),w^*)\longrightarrow (B(X^{**}),w)$ is nowhere continuous. As before, let  $X^*=C(K)$ for some compact Hausdorff space $K$. Fix $k\in K$. We have $C(\Omega)^*=\text{span}\{\delta(k)\}\oplus _1 N$ for some closed subspace $N$. Thus $C(K)^{**}=\text{span}\{\delta(k)\}^{\perp}\oplus _{\infty} N^{\perp}$. Hence $N^{\perp}$ is an one-dimensional $M$-summand. Let $N^{\perp}=\text{span}\{f\}$, for some $f\in C(K)^{**}$. It is well known that $C(K)^{**} = C(K')$ for some compact space $K'$. Hence by \cite[Example~1.4.(a)]{HWW}, we have that $f=\chi_{\{k_0\}}$, for some isolated point $k_0\in K'$. Now since $C(K')^\ast = \text{span}\{\delta(k_0)\}\oplus_1 N'$ for some $w^*$-closed subspace $N' \subset C(K')$, it is easy to see by arguments indicated in Theorem \ref{t31} that $\delta(k_0)$ is a point of continuity of $id: (B(X^{(IV)}),w^*)\longrightarrow (B(X^{(IV)}),\|.\|)$. Hence it is a point of continuity of $id: (B(X^{(IV)}),w^*)\longrightarrow (B(X^{(IV)}),w)$, which is an extreme point.
\end{proof}

\section{$G_{\dd}$ character of the set of points of continuity}\label{sect-g-delta}
We next investigate Banach space $X$ with some geometric structure such that the set of points of continuity of $id: (B(X^{***}),w^*)\longrightarrow (B(X^{***}),w)$ (respectively, $id: (B(X^{***}),w^*)\longrightarrow (B(X^{***}),\nn)$) forms a dense $G_\delta$ subset of $(S(X^{***}),w^*)$.

We do not know examples of nonreflexive  Banach spaces $X$ for which $S(X^{***})$ is a dense $G_\delta$ subset of $(S(X^{(V)}),w^*)$.

\medskip

Let $X$ be Banach space. Put
$$
O_n:=X^*\setminus \left(1-\frac{1}{n}\right)B(X^*),\ \hbox{for }\nin.
$$
Obviously, $O_n$ is $w^*$-open in $X^*$, for $\nin$. It is plain that
\begin{equation}\label{eqSXstar}
S(X^*)=B(X^*)\cap \bigcap_{n\in \mathbb{N}} O_n,
\end{equation}
hence $S(X^*)$ is a (dense) $G_\delta$ subset of $(B(X^*),w^*)$. Notice, too, that $S(X^*)$ is dense in $(B(X^{***}),w^*)$, in particular in $(S(X^{***}),w^*)$, as it follows from Goldstine theorem.

Let $Q$ be the canonical projection associated to the decomposition $X^{***}=X^*\oplus X^{\perp}$.
Put
\begin{equation}\label{eq-s}
{\mathcal S}:= \{x^{***}\in S(X^{\ast\ast\ast}):\  Q(x^{***}) \in S(X^\ast)\}.
\end{equation}
In other words, $\mathcal S$ is the set of all Hahn--Banach extensions of elements in $S(X^*)$. It follows easily from (\ref{eqSXstar}) that $\mathcal S=B(X^{***})\cap \bigcap_{\nin}Q^{-1}(O_n)\ (\subset S(X^{***}))$, hence $\mathcal S$ is a $G_{\dd}$-subset of $(B(X^{***}),w^*)$, in particular of $(S(X^{***}),w^*)$.

The following result has a simple proof:

\begin{prop}\label{prop-s}
A Banach space $X$ is Hahn--Banach smooth if and only if $\mathcal S=S(X^*)$.
\end{prop}
\begin{proof} We shall have in mind that $X^{***}=X^*\oplus X^{\perp}$, and so that $Q:X^{***}\longrightarrow X^{*}$ is the canonical projection associated to this decomposition.
Assume first that $X$ is Hahn--Banach smooth. Let $x^{***}\in \mathcal S$. Thus, $Q(x^{***})\in S(X^*)$, and $Q(x^{***})$ has a unique Hahn--Banach extension to $X^{**}$. It follows that $x^{***}=Q(x^{***})\in S(X^*)$. On the other hand, if $x^*\in S(X^*)$ then $x^*$ is the unique Hahn--Banach extension to $X^{**}$, so $x^*\in \mathcal S$.

Assume now that $\mathcal S=S(X^*)$. Let $x^*\in S(X^*)$. If $y^{***}$ is a different Hahn--Banach extension of $x^*$ to $X^{**}$, then $y^{***}\in\mathcal S$. Moreover, $y^{***}=x^*+x^{\perp}$, where $x^{\perp}\in X^{\perp}$ and, since $y^{***}\not=x^*$, it happens that $x^{\perp}\not=0$. Thus, $y^{***}\not\in X^*$, and this is a contradiction. This shows that $X$ has property-(U) in $X^{**}$.
\end{proof}

In view of Proposition \ref{prop-s} and the comments preceding it, the proof of the following result should be clear (notice that $S(X^*)$ is always $w^*$-dense in $S(X^{***})$ as soon as $X$ is infinite-dimensional):

\begin{prop}\label{5t1}
    Let $X$ be a nonreflexive Hahn--Banach smooth space. Then, $S(X^*)$ is a dense $G_\delta$ subset of $(S(X^{***}),w^*)$.
\end{prop}

\begin{rem} \rm It was noticed in \cite{B} that the Hahn--Banach smoothness property is preserved by closed subspaces, quotients, and $c_0$-sums. We do not know if the property `$S(X^\ast)$ is a dense $G_{\delta}$ subset of $(S(X^{***}),w^*)$' has a similar behaviour.
\end{rem}

It follows from \cite[Theorem III.4.6.(e)]{HWW} that any nonreflexive $M$-embedded space can be renormed so that $X^*$ is strictly convex and $X$ is again an $M$-embedded space in the new norm. Thus these examples satisfy the hypothesis of the following easy corollary, whose proof, that will be omitted,  uses Lemma \ref{lemma-preserved} above.
\begin{cor}\label{c46}
	 Let $X$ be a nonreflexive Hahn--Banach smooth space whose dual is strictly convex. Then, $S(X^*)$ is the set of extreme points of $B(X^{***})$ that are points of $w^*$-$w$-continuity. These form a $G_\delta$ dense subset of $(S(X^{***}),w^*)$.
	 \end{cor}
\begin{rem}\rm
Most often in applications it is enough to consider `richness' of points of $w^\ast$-$w$-continuity in $B(X^*)$ that are also norm attaining. This set is, in general, strictly smaller that the set of all points of $w^*$-$w$-continuity in $B(X^*)$. For example, let $X$ be a nonreflexive Hahn--Banach smooth Banach space. Then, the set of all points of $w^*$-$w$-continuity of $B(X^*)$ is precisely $S(X^*)$ (see Theorem \ref{thm-hb-smooth}). However, not every point in $S(X^*)$ is norm-attaining, in view of James' compactness theorem. On the other hand, if $X$ is a nonreflexive M-embedded Banach space, then the set of all points of $w^*$-$w$-continuity of $B(X^{***})$ is, according to Theorem \ref{t31}, precisely $S(X^*)$, and all of them are clearly in $NA(X^{***})$. In this case, then, the set of all $w^*$-$w$-continuity points of $B(X^{***})$ that are also norm-attaining is a dense  $G_{\dd}$  subset of $(S(X^{***}),w^*)$, according to Proposition \ref{5t1} above.
We do not know of other geometric conditions (in the nonreflexive case) that may ensure a similar behaviour.

The analysis of the set of points of $w^*$-$w$-continuity that are also norm-attaining through the linear isometric identifications as done in Example \ref{ex-c} must be done carefully. For instance, (and we keep the notation there), $c_0$ is M-embedded, hence, in particular, the element $\Lambda:=(1/2^n)\in S(\ell_1)$ is of $w^*$-$w$-continuity in $B(c_0^*)\ (=B(\ell_1))$, and is clearly not norm-attaining. However, $T(\Lambda)\ (\in S(c^*))$ satisfies $\langle {\tt 1},T(\Lambda)\rangle=1$ (where ${\tt 1}\in c$ is the constant 1 vector), so it attains its norm (and it is, to be sure, a point of $w^*$-$w$-continuity). The reason for this odd behaviour is, of course, the fact that $T:\ell_1\longrightarrow c^*$ is not $w^*$-$w^*$-continuous. By the way, the set of points of $w^*$-$w$-continuity of $B(c_0^*)\ (=B(\ell_1))$ is $S(c_{00})$ (a dense $F_{\sigma}$ subset of $(S(\ell_1),\nn_1)$), where $c_{00}$ is the space of all finitely supported vectors in $c_0$.
\end{rem}
The last result deals with spaces $X$ that satisfy the stronger property that the $w^*$ and norm topologies coincide on $S(X^*)$. We recall that when ${\mathcal K}(X)$ is an $M$-ideal in ${\mathcal L}(X)$, then $X$ is an $M$-embedded space and the weak$^*$, weak and norm topologies coincide on $S(X^*)$ (see Propositions VI.4.4 and VI.4.6 in \cite{HWW}).
\begin{cor}\label{5t3}
    Let $X$ be a nonreflexive $M$-embedded space such that on $S(X^*)$ both the $w^*$ and the $\nn$ topology coincide. Then, the set of points of continuity of $id: (B(X^{***}),w^*)\longrightarrow (B(X^{***}),\|.\|)$ forms a $G_\delta$ dense subset of $(S(X^{***}),w^*)$.
\end{cor}
\begin{proof}
      It is easy to prove that the set of points of continuity of $id:(B(X^{***}),w^*)\to (B(X^{***}),\nn)$ is just $S(X^*)$. By Theorem~\ref{thm-hb-smooth}, $X$ is a Hahn--Banach smooth space. The result now follows from Theorem~\ref{5t1}.
\end{proof}

{\bf Conflict of interest:} The authors declare that there is no conflict of interest.

\bibliographystyle{plain, abbrv}

\section*{Acknowledgments}
This work is part of the project "Classification of Banach spaces using differentiability", funded by the Anusandhan National Research Foundation (ANRF), Core Research Grant, CRG2023-000595. The first author was a postdoctoral fellow in this project. The second named author was partially supported by Project  PID2021-122126NB-C33 (Plan Estatal de
 Investigaci\'on Cient\'{\i}fica, T\'ecnica y de Innovaci\'on 2021-2023, Spain) and the Universitat Polit\`ecnica de Val\`encia (Spain).
\end{document}